\documentclass[12pt,twoside]{article} 
\usepackage{amsmath, amsthm}
\usepackage[dvips]{graphicx}
\usepackage{amsfonts}
\usepackage{amssymb}

\date{} 

\begin{document}

\centerline {\Large{\bf A REPRESENTATION OF TWISTED GROUP}} 

\centerline{} 

\centerline{\Large{\bf ALGEBRA OF SYMMETRIC GROUPS ON}} 

\centerline{} 

\centerline{\Large{\bf WEIGHT SUBSPACES OF FREE}} 

\centerline{} 

\centerline{\Large{\bf ASSOCIATIVE COMPLEX ALGEBRA}} 

\centerline{} 

\centerline{\bf {Milena So\v{s}i\'{c}}} 

\centerline{} 

\centerline{Department of Mathematics} 

\centerline{University of Rijeka} 

\centerline{Trg bra\'{c}e Ma\v{z}urani\'{c}a 10, Rijeka 51000, CROATIA} 

\centerline{} 

\centerline{e-mail:  msosic@math.uniri.hr} 

\centerline{} 

\newtheorem{Theorem}{\quad Theorem}[section] 

\newtheorem{Definition}[Theorem]{\quad Definition} 

\newtheorem{Corollary}[Theorem]{\quad Corollary} 

\newtheorem{Proposition}[Theorem]{\quad Proposition} 

\newtheorem{Lemma}[Theorem]{\quad Lemma} 

\newtheorem{Example}[Theorem]{\quad Example} 

\newtheorem{Remark}[Theorem]{\quad Remark} 

\begin{abstract} 
Here we consider two algebras, a free unital associative complex algebra (denoted by ${\mathcal{B}}$) equiped with a multiparametric \textbf{\emph{q}}-differential structure  and a twisted group algebra (denoted by ${\mathcal{A}(S_{n})}$), with the motivation to represent the algebra ${\mathcal{A}(S_{n})}$ on the (generic) weight subpaces of the algebra ${\mathcal{B}}$. One of the fundamental problems in ${\mathcal{B}}$ is to describe the space of all constants (the elements which are annihilated by all multiparametric partial derivatives). To solve this problem, one needs some special matrices and their factorizations in terms of simpler matrices. A simpler approach is to study first certain canonical elements in the twisted group algebra ${\mathcal{A}(S_{n})}$. Then one can use certain natural representation of ${\mathcal{A}(S_{n})}$ on the weight subspaces of ${\mathcal{B}}$, which are the subject of this paper.  
\end{abstract}



\section{Introduction}     

First we recall free unital associative complex algebra ${\mathcal{B}}$ generated by $N$ generators ${\{e_{i}\}_{i\in{\mathcal{N}}}}$ (each of degree one, ${N=\textrm{Card}~{\mathcal{N}}}$ )  with a multiparametric \textbf{\emph{q}}$(=\{q_{ij}\})$-differential structure such that \textbf{\emph{q}}-differential operators ${\{\partial_{i}\}_{i\in{\mathcal{N}}}}$ act on ${\mathcal{B}}$ according to the twisted Leibniz rule:\\ 
${\partial_{i}(e_{j}x)}=\delta_{ij}x+q_{ij}e_{j}\partial_{i}(x) \textrm{, for each\, } x\in{\mathcal{B}} \textrm{\, with\, } {\partial_{i}(1)}=0 \textrm{\, and\, } {\partial_{i}(e_{j})}=\delta_{ij}$ (where $q_{ij}$'s are given complex numbers). 
The algebra ${\mathcal{B}}$ is graded by total degree and more generally it could be considered as multigraded, because the operators $\partial_{i}$ (of degree ${-1}$) respect the (direct sum) decomposition of ${\mathcal{B}}$ into multigraded subspaces ${\mathcal{B}}_{Q}$.\\
We also recall a twisted group algebra ${\mathcal{A}(S_{n})}$ of the symmetric groups $S_{n}$ with coefficients in the polynomial ring $R_{n}$ in $n^2$ commuting variables $X_{a\,b}$, $1\le a,b \le n$ (c.f \cite{10} for details). The algebra ${\mathcal{A}(S_{n})}$ is defined as a semidirect product of $R_{n}$ and  the usual group algebra ${\mathbb{C}}[S_n]$ of the symmetric group $S_n$. The elements of ${\mathcal{A}(S_{n})}$ are the linear combinations $\sum_{g_{i}\in S_{n}} p_{i}\,g_{i}$ with $p_{i}\in R_{n}$. The multiplication in ${\mathcal{A}(S_{n})}$ is given by the formula (\ref{mtwga}) in the Section 3, where we impose that $S_{n}$ naturally acts on $R_{n}$.\\ 
In Section 4 we define a natural representation of the twisted group algebra ${\mathcal{A}(S_{n})}$, $n\ge 0$ on the weight subspaces ${\mathcal{B}}_{Q}$ ($\textrm{Card}~Q=n$) of ${\mathcal{B}}$. In this representation some factorizations of certain canonical elements from ${\mathcal{A}(S_{n})}$ will immediately give the corresponding matrix factorizations and also determinant factorizations.

\section{Free unital associative ${\mathbb{C}}$-algebra}

We recall first a free (unital) associative ${\mathbb{C}}$-algebra (see also \cite{9}), denoted by $\mathcal{B} = {\mathbb{C}}\left< e_{i_{1}},\dots, e_{i_{N}} \right>$, where ${\mathcal{N}}=\{i_{1},\dots, i_{N}\}$ is a fixed subset of the set ${{\mathbb{N}}_{0}=\{0,1,\dots \}}$ of nonnegative integers.
The algebra ${\mathcal{B}}$ is generated by $N$ generators ${\{e_{i}\}_{i\in{\mathcal{N}}}}$ (each of degree one), so we can think of ${\mathcal{B}}$ as an algebra of noncommutative polynomials in $N$ noncommuting variables ${e_{i_{1}},\dots, e_{i_{N}}}$.
The algebra ${\mathcal{B}}$ is naturally graded by the total degree $\displaystyle{\mathcal{B}}=\bigoplus_{n \ge 0}{\mathcal{B}}^{n}$, where ${{\mathcal{B}}^{0}={\mathbb{C}}}$ and ${{\mathcal{B}}^{n}}$ consists of all homogeneous polynomials of total degree $n$ in variables ${e_{i_{1}},\dots,e_{i_{N}}}$.\\ 
In view of the fact that every sequence ${l_{1},\dots, l_{n}\in{\mathcal{N}}}$, ${l_{1}\le \dots \le l_{n}}$ can be thought of as a multiset ${Q=\{l_{1}\le \dots \le l_{n}\}}$ over ${\mathcal{N}}$ of size ${n=\textrm{Card}~Q}$ (cardinality of $Q$), we see that each weight subspace ${{\mathcal{B}}_{Q}={\mathcal{B}}_{l_{1}\dots l_{n}}}$, corresponding to a multiset $Q$, is given by 
\begin{equation}\label{BQ}
{\mathcal{B}}_{Q}=\textrm{\emph{span}}_{\mathbb{C}} \left\{e_{j_{1}\dots j_{n}}:=e_{j_{1}}\cdots e_{j_{n}}\mid {j_{1}\dots j_{n}}\in \widehat{Q} \right\},
\end{equation}
where $\widehat{Q}$ denotes the set of all distinct permutations of the multiset $Q$. It follows that ${\dim{\mathcal{B}}_{Q}=\textrm{Card}~\widehat{Q}}$. Thus, we obtain a finer decomposition of ${\mathcal{B}}$ into multigraded components ($=$ weight subspaces):\, $\displaystyle {\mathcal{B}}=\bigoplus_{n \ge 0, \, l_{1}\leq \dots \leq l_{n},\, l_{j}\in{\mathcal{N}}}{\mathcal{B}}_{l_{1}\dots l_{n}}.$ 
Note that if ${l_{1}<\dots <l_{n}}$, then $Q$ is a set and the corresponding weight subspace ${{\mathcal{B}}_{Q}}$ we call \emph{generic}. Other (nongeneric) weight subspaces ${{\mathcal{B}}_{Q}}$ we call \emph{degenerate}.\\ 
We can interpret $\textrm{\textbf{\emph{q}}}$ as a map ${\textrm{\textbf{\emph{q}}}\colon {\mathcal{N}}\times {\mathcal{N}}\to {\mathbb{C}}}$, ${\left( i,j \right)\mapsto q_{ij}}$ for all ${i,j\in{\mathcal{N}}}$. Then, on the algebra ${\mathcal{B}}$, we introduce $N$ linear operators ${\partial_{i}=\partial_{i}^{\textrm{\textbf{\emph{q}}}}\colon {\mathcal{B}} \to {\mathcal{B}}}$, ${i\in{\mathcal{N}}}$, defined recursively, as follows:
$${\partial_{i}(1)}=0,\hspace{10pt} {\partial_{i}(e_{j})}=\delta_{ij},$$
$${\partial_{i}(e_{j}x)}=\delta_{ij}x+q_{ij}e_{j}\partial_{i}(x) \hspace{10pt} \textrm{for each}\hspace{5pt} {x\in{\mathcal{B}}},\hspace{3pt} i,j\in{\mathcal{N}}.$$ 
($\delta_{ij}$ is a standard Kronecker delta i.e $\delta_{ij}=1$ if $i=j$, and $0$ otherwise.)\\
These \textbf{\emph{q}}-differential operators ${\{\partial_{i}\}_{i\in{\mathcal{N}}}}$ act as a generalized \emph{i}-th partial derivatives on the algebra ${\mathcal{B}}$; they depend on additional parameters (complex numbers) $q_{ij}$. Therefore, we can say that $\partial_{i}^{\textrm{\textbf{\emph{q}}}}$ is a multiparametrically deformed \emph{i}-th partial derivative or shortly \textbf{\emph{q}}-deformed \emph{i}-th partial derivative. It is easy to see that if all $q_{ij}$'\emph{s} are equal to one, then $\partial_{i}^{\textrm{\textbf{\emph{q}}}}$ coincides with a usual \emph{i}-th partial derivative $\partial_{i}$.\\
\noindent In what follows we will consider ${\mathcal{B}}$  with this `\textbf{\emph{q}}-differential structure'.\\

Let us denote by ${{\mathfrak{B}_{Q}}=\left\{ e_{\underline{j}}\mid {\underline{j}}\in \widehat{Q} \right\}}$ the monomial basis of ${{\mathcal{B}}_{Q}}$, where $\underline{j}:=j_{1}\dots j_{n}$. Then the action of ${\partial_{i}}=\partial_{i}^{\textrm{\textbf{\emph{q}}}}$ on a typical monomial ${e_{\underline{j}}\in {\mathfrak{B}_{Q}}}$ is given explicitly by the formula:
\begin{equation}\label{pardeg}
{\partial_{i}}(e_{\underline{j}})=\sum_{1\le k\le n, \, j_k=i} q_{ij_{1}}\cdots q_{ij_{k-1}}\,e_{j_{1}\dots \widehat{j_{k}}\dots j_{n}},
\end{equation}
where $\widehat{j_{k}}$ denotes the omission of the corresponding index ${j_{k}}$.\\
The number of terms in this sum is equal to the number of appearances (multiplicity) of the generator $e_{i}$ in the monomial $e_{\underline{j}}$.\\ 

\noindent In the generic case, when $Q$ is a set, the formula (\ref{pardeg}) is reduced to:
\begin{equation}\label{partgen}
 {\partial_{i}}(e_{\underline{j}})=q_{ij_{1}}\cdots q_{ij_{k-1}}\,e_{j_{1}\dots \widehat{j_{k}}\dots j_{n}}.
\end{equation}
If ${j_{k}=i}$ for all ${1\leq k\leq n}$, then from (\ref{pardeg}) we get the following important special case
\begin{equation}\label{paren}
{\partial_{i}(e_{i}^{n})}=\left( 1+q_{ii}+q_{ii}^2+\cdots +q_{ii}^{n-1}\right)\,e_{i}^{n-1}=\left[ n \right]_{q_{ii}}\,e_{i}^{n-1},
\end{equation}
where ${\left[ n \right]_{q}=1+q+\cdots +q^{n-1}}$ is a \emph{q}-analogue of a natural number $n$. For $q_{ii}=1$ the formula (\ref{paren}) can be read
as the classical formula \, ${{\partial_{i}(e_{i}^{n})}=n\cdot e_{i}^{n-1}}$.\\    

\noindent For ${x\in {\mathcal{B}}_{l_{1}\dots l_{n}}}$ and ${y\in {\mathcal{B}}}$ we have a more general formula: 
$${\partial_{i}}(xy)={\partial_{i}}(x)\,y + q_{il_{1}}\cdots q_{il_{n}}\,x\,{\partial_{i}}(y)\hspace{10pt}\textrm{for each}\hspace{5pt} {i\in{\mathcal{N}}}.$$
On the other hand, with the motivation of treating better the matrices of $\partial_{i}\vert_{{\mathcal{B}}_{Q}}$, we introduce a multidegree operator ${\partial \colon {\mathcal{B}} \to {\mathcal{B}}}$ with\, $\displaystyle{\partial=\sum_{i\in{\mathcal{N}}} {e_{i\, }\partial_{i}}}$,\, where ${e_{i}\colon {\mathcal{B}} \to {\mathcal{B}}}$ are considered as (multiplication by ${e_{i}}$) operators on ${\mathcal{B}}$. 
The operator ${\partial}$ preserves the direct sum decomposition of the algebra ${\mathcal{B}}$, i.e each subspace ${\mathcal{B}}_{Q}$ is an invariant subspace of ${\partial}$. Moreover, we denote by ${\partial^Q \colon {\mathcal{B}}_{Q} \to {\mathcal{B}}_{Q}}$ the restriction of ${\partial \colon {\mathcal{B}} \to {\mathcal{B}}}$ to the subspace ${\mathcal{B}}_{Q}$ (i.e\, $\displaystyle\partial^Q {x}=\partial {x}$ \, for every\, $x \in {\mathcal{B}}_{Q}$). Hence for each ${j_{1}\dots j_{n}}\in \widehat{Q}$ we get
\begin{align*}
{\partial^Q}&\left ( e_{j_{1}\dots j_{n}}\right )=\sum_{i\in{\mathcal{N}}} {e_{i}\partial_{i}} \left ( e_{j_{1}\dots j_{n}}\right )
=\sum_{i\in{\mathcal{N}}} e_{i} \sum_{1\le k\le n, \, j_{k}=i} q_{ij_{1}}\cdots q_{ij_{k-1}}\,e_{j_{1}\dots \widehat{j_{k}}\dots j_{n}}\\
&=\sum_{1\le k\le n} \, \sum_{i\in{\mathcal{N}}, \, i=j_{k}} q_{ij_{1}}\cdots q_{ij_{k-1}}\,e_{ij_{1}\dots \widehat{j_{k}}\dots j_{n}} = \sum_{1\le k\le n} q_{j_{k}j_{1}}\cdots q_{j_{k}j_{k-1}}\,e_{j_{k}j_{1}\dots \widehat{j_{k}}\dots j_{n}}.
\end{align*}
If ${\textrm{B}_Q}$ denotes the matrix of ${\partial^Q}$ w.r.t basis ${\mathfrak{B}_{Q}}$ (totally ordered by the Johnson-Trotter ordering on permutations c.f~\cite{11}) of ${{\mathcal{B}}_{Q}}$, then we can write
\begin{equation}\label{partQ}
{\textrm{B}_Q}\,e_{j_{1}\dots j_{n}}=\sum_{1\le k\le n} q_{j_{k}j_{1}}\cdots q_{j_{k}j_{k-1}}\,e_{j_{k}j_{1}\dots \widehat{j_{k}}\dots j_{n}}.
\end{equation}
Note that for any multiset ${Q=\left\{ k_{1}^{n_{1}},\dots, k_{p}^{n_{p}} \right\}}$ ($k_{i}$ distinct) of cardinality $n$ (where $n=n_{1}+\cdots +n_{p}$)\, the size of the matrix ${\textrm{B}_Q}$ is equal to the following multinomial coefficient 
$${\frac{n!}{n_{1}! \cdots n_{p}!}=\binom{n}{n_{1}, \dots , n_{p}} \, (=\dim{\mathcal{B}}_{Q})}.$$
The entries of ${\textrm{B}_Q}$ are polynomials in $q_{ij}$'\emph{s}, hence its determinant is also a polynomial in $q_{ij}$'\emph{s}. Clearly, in the generic case (i.e $Q$ is a set) the entries of ${\textrm{B}_Q}$ are monomials in $q_{ij}$'\emph{s} and the size of ${\textrm{B}_Q}$ is equal to $n!$.\\
In the algebra ${\mathcal{B}}$ of particular interest are the elements called constants. They are by definition annihilated by all multiparametric partial derivatives. In other words, an element $C\in {\mathcal{B}}$ is called a constant in ${\mathcal{B}}$ if ${\partial_{i}(C)=0}$\, for every ${i\in{\mathcal{N}}}$. It is obvious that \, $\displaystyle{\partial{C}=\sum_{i\in{\mathcal{N}}} {e_{i\, }}\partial_{i\, }C=0}$\, iff\, ${\partial_{i\, }C=0}$\, for all $i\in{\mathcal{N}}$.\\

\noindent We denote by ${\mathcal{C}}$ the space of all constants in ${\mathcal{B}}$ and similarly by ${{\mathcal{C}}_{Q}}$ the space of all constants belonging to ${{\mathcal{B}}_{Q}}$.
Then ${{\mathcal{C}}=\ker{\partial}}$ (where ${\ker{\partial}}$ denotes the kernel of the multidegree operator ${\partial}$). By using the fact that the operator ${\partial}$ preserves the direct sum decomposition of the algebra ${\mathcal{B}}$, we have that the space ${{\mathcal{C}}}$ inherits the direct sum decomposition into multigraded subspaces ${{\mathcal{C}}_{Q}}$. Hence the fundamental problem of determining the space ${\mathcal{C}}$ can be reduced to determining all finite dimensional spaces ${{\mathcal{C}}_{Q}}$ ($={\ker{\partial^Q}}$) for all multisets $Q$ over ${\mathcal{N}}$. Thus of particular interest is the study of $\det {\textrm{B}_Q}$. The special role play the actual values of the parameters $q_{ij}$'\emph{s} $($called singular values or singular parameters$)$ for which ${\det{\textrm{B}_{Q}}}$ vanishes. In Section 4 a formula for the factorization of the matrix ${\textrm{B}_Q}$ and its determinant will be given.

\section{A twisted group algebra of the symmetric group}

\noindent Here, we firstly recall some basic factorizations in the twisted group algebra $ {\mathcal{A}(S_{n})}$ of the symmetric group
$S_{n}$ with coefficients in a polynomial algebra in commuting variables $X_{a\,b}$ ($1\le a,b\le n$).\\ 
Secondly, (in the next section) we will consider natural representation of the algebra $ {\mathcal{A}(S_{n})}$ on the weight subspaces ${{\mathcal{B}}_{Q}}$ of ${\mathcal{B}}$.\\

\noindent In \cite{10} we have introduced a twisted group algebra ${\mathcal{A}(S_{n})}$ of the symmetric group $S_{n}$, which we recall now.\\ 
The elements of the twisted group algebra ${\mathcal{A}(S_{n})}=R_{n}\rtimes {\mathbb{C}}[S_n]$ are the linear combinations $\displaystyle\sum_{g_{i}\in S_{n}} p_{i}\,g_{i}$, where the coefficients are polynomials $p_{i} = p_{i}(\dots, X_{a\,b},\dots)$ in commuting variables $X_{a\,b}$, $1\le a,b\le n$. The multiplication in ${\mathcal{A}(S_{n})}$ is given by the formula 
\begin{equation}\label{mtwga} 
        (p_{1}g_{1}) \cdot ( p_{2}g_{2})=(p_{1}\cdot (g_{1}{\bf.}p_{2}))\, g_{1}g_{2},
\end{equation}  
where \, $g_{1}{\bf.}p_{2}=g_{1}{\bf.}p_{2}(\dots, X_{a\,b},\dots)=p_{2}(\dots, X_{g_{1}(a)\,g_{1}(b)},\dots)\, g_{1}$.\\
The algebra ${\mathcal{A}(S_{n})}$ is associative but not commutative (i.e\, $g{\bf.}p \neq p\, g$).\\
In the algebra ${\mathcal{A}(S_{n})}$ we have introduced more specific elements, denoted by ${g^*}$ and defined by 
\begin{equation}\label{gtilda} 
         g^*:=\left(\prod_{(a,b)\in I(g^{-1})} X_{a\,b} \right) g
\end{equation}
for every $g\in S_{n}$, where ${I(g)=\{(a,b)\mid 1\le a < b\le n, \, g(a) > g(b) \}}$ denotes the set of inversions of the permutation $g$.\\ 
Of particular interest are the elements $t_{b,a}^*\in {\mathcal{A}(S_{n})}$, $1\le a\le b\le n$, where for $a<b$, $t_{b,a}\in S_{n}$ denotes the inverse of the cyclic permutation $t_{a,b}\in S_{n}$ i.e
\begin{equation*} 
t_{b,a}=\left( \begin{array}{ccccccccccc}
                 1 & \cdots & a-1 & a & a+1 & \cdots & b-1 & b & b+1 & \cdots & n \\
                 1 & \cdots & a-1 & a+1 & a+2 & \cdots & b & a & b+1 & \cdots & n \\
              \end{array}\right)
\end{equation*}
(see also \cite{3}).
Every permutation $g\in S_{n}$ can be decomposed into cycles (from the left) as follows:
\begin{equation}\label{gdecomp} 
         g =t_{k_{n},n}\cdot t_{k_{n-1},n-1}\cdots t_{k_{j},j}\cdots t_{k_{2},2}\cdot t_{k_{1},1}\, \left(= \prod_{1\le j\le n}^{\gets} t_{k_{j},j} \right). 
\end{equation}
By applying (\ref{mtwga}) we obtain the following formula
\begin{equation}\label{pg1g2} 
g_{1}^*\cdot g_{2}^*=X(g_{1}, g_{2})\, (g_{1}g_{2})^*,
\end{equation}
where the multiplication factor $X(g_{1}, g_{2})$ takes care of the reduced number of inversions in the group product of $g_{1}, g_{2} \in S_{n}$ and it is given by
 $$X(g_{1}, g_{2})=\prod_{(a,b)\in I(g_{1}^{-1})\backslash I((g_{1}g_{2})^{-1})} X_{\{a,\,b\}} \left(=\prod_{(a,b)\in I(g_{1})\cap I(g_{2}^{-1})} X_{\{g_{1}(a),\,g_{1}(b)\}} \right).$$
By studying in details the elements $g^*\in {\mathcal{A}(S_{n})}$ and more specific properties arising from (\ref{pg1g2}) and (\ref{gdecomp}) (see also \cite{3} and \cite{4}) we came to the following conclusions. \footnote{For more details c.f the  Ph.D. thesis \cite{7}.}\\
 
\noindent Let \, $\displaystyle{\alpha^*_{n}} := \sum_{g\in S_{n}} g^*$,\, $n\ge 1$.\\
\begin{enumerate} 
\item If we define simpler elements ${\beta^*_{k}}\in {\mathcal{A}(S_{n})}$ ($1\le k\le n$) as follows
\begin{align*} 
{\beta^*_{n}}&= t_{n,1}^* + t_{n-1,1}^* + \cdots + t_{2,1}^* + t_{1,1}^* ,\\
{\beta^*_{n-1}}&= t_{n,2}^* + t_{n-1,2}^* + \cdots + t_{3,2}^* + t_{2,2}^*,\\
\vdots\\
{\beta^*_{n-k+1}}&= t_{n,k}^* + t_{n-1,k}^* + \cdots + t_{k+1,k}^* + t_{k,k}^*,\\
\vdots\\
{\beta^*_{2}}&= t_{n,n-1}^* + t_{n-1,n-1}^*,\\
{\beta^*_{1}}&= t_{n,n}^*\, ( = id)
\end{align*}
then the element ${\alpha^*_{n}}\in {\mathcal{A}(S_{n})}$ can be decomposed into the product of elements ${\beta^*_{k}}$ ($1\le k\le n$) i.e we have
\begin{equation}\label{alphafac}
{\alpha^*_{n}} = {\beta^*_{1}} \cdot {\beta^*_{2}} \cdots {\beta^*_{n}} = \prod_{1\le k\le n} {\beta^*_{k}}.
\end{equation}
\item Now we define yet simpler elements in ${\mathcal{A}(S_{n})}$ for all $1\le k\le n-1$
\begin{align*}
{\gamma^*_{n-k+1}}&=\left( id-t_{n,k}^*\right)\cdot \left( id-t_{n-1,k}^*\right)\cdots
\left( id-t_{k+1,k}^*\right),\\
{\delta^*_{n-k+1}}&=\left( id-(t_{k}^*)^{2} \, t_{n,k+1}^*\right)\cdot \left( id-(t_{k}^*)^{2} \, t_{n-1,k+1}^*\right)\cdots
\left( id-(t_{k}^*)^{2} \, t_{k+1,k+1}^*\right),
\end{align*}
where\, $t_{k+1,k+1}^* = id$ and
\begin{equation}\label{sqta}
(t_{k}^*)^{2} = X_{\{k,\,k+1\}}\, id
\end{equation}
because 
$$(t_{k}^*)^{2} = (X_{k\,k+1}\, t_{k})\cdot (X_{k\,k+1}\, t_{k}) = X_{k\,k+1}\cdot X_{k+1\,k}\, (t_{k})^{2} = X_{\{k,\,k+1\}}\, id$$ 
with $t_k = t_{k+1,k}\,(=t_{k,k+1}).$\\

\noindent Then each element ${\beta^*_{k}}$, $2\le k\le n$ can be further factored as
\begin{equation}\label{betafac}
{\beta^*_{k}} = {\delta^*_{k}} \cdot \left( {\gamma^*_{k}}\right)^{-1}.
\end{equation}
\end{enumerate}

\section{Representation of the twisted group algebra ${\mathcal{A}(S_{n})}$ on the subspaces ${\mathcal{B}_{Q}}$ of ${\mathcal{B}}$}

\noindent Recall that ${\mathcal{B}}_{Q}=\textrm{\emph{span}}_{\mathbb{C}} \left\{e_{j_{1}\dots j_{n}}=e_{j_{1}}\cdots e_{j_{n}}\mid {j_{1}\dots j_{n}}\in \widehat{Q} \right\}$ denotes the weight subspace of the free associative complex algebra ${\mathcal{B}}$, where $\widehat{Q}$ is the set of all distinct permutations of the multiset $Q$ (see Section 2).\\ 

\noindent Let $V$ be a vector space over a field $F$. Then $\textrm{End}(V)$ denotes the algebra of all endomorphisms of $V$. If we denote by $A$ an associative algebra then by definition a representation of $A$ on $V$ is any algebra homomorphism $\varphi \colon A \to \textrm{End}(V)$.\\

\noindent Our next task is to define a representation (see the formula (\ref{ro}) below) $\varrho \colon {\mathcal{A}(S_{n})} \to \textrm{End}({\mathcal{B}}_{Q})$,\, where ${\mathcal{B}}_{Q}$ is defined by (\ref{BQ}).\\

\noindent We first recall that ${R_n={\mathbb{C}}[X_{a\,b} \mid 1\le a,b\le n ]}$ denotes the polynomial ring with unit element ${1\in R_n}$ and ${{\mathbb{C}}[S_n]=\left\{\sum_{\sigma\in S_{n}} c_{\sigma} \sigma \mid c_{\sigma}\in {\mathbb{C}}\right\}}$ denotes the usual group algebra. In ${\mathbb{C}}[S_n]$ the multiplication is given by
$$\left( \sum_{\sigma\in S_{n}} c_{\sigma} \sigma \right)\cdot \left( \sum_{\tau\in S_{n}} d_{\tau} \tau \right)= \sum_{\sigma, \tau\in S_{n}} (c_{\sigma}d_{\tau})\, \sigma \tau.$$
Recall that ${\mathcal{A}(S_{n})}=R_{n}\rtimes {\mathbb{C}}[S_n]$, so we first consider the representations $\varrho_1$ of $R_n$ and $\varrho_2$ of ${\mathbb{C}}[S_n]$:
$$\varrho_1 \colon R_{n} \to \textrm{End}({\mathcal{B}}_{Q}),$$
$$\varrho_2 \colon {\mathbb{C}}[S_n] \to \textrm{End}({\mathcal{B}}_{Q})$$
given by Definition~\ref{ro1} and Definition~\ref{ro2} respectively.\\

\noindent Let $Q_{a\,b}$, $1\le a,b \le n$ denote a diagonal operator on ${\mathcal{B}_{Q}}$ (see (\ref{BQ})) defined by
\begin{equation}\label{Qab} 
         Q_{a\,b}\,e_{j_{1}\dots j_{n}}:= q_{j_{a}j_{b}}\,e_{j_{1}\dots j_{n}}.
\end{equation}
Note that
$$Q_{a\,b}\cdot Q_{c\,d} = Q_{c\,d}\cdot Q_{a\,b}.$$

\begin{Definition}\label{ro1}
We define a representation $\varrho_1 \colon R_{n} \to \textnormal{End}({\mathcal{B}}_{Q})$ on the generators $X_{a\,b}$ by the formula
$$\varrho_1(X_{a\,b}):=Q_{a\,b} \hspace{25pt} 1\le a,b \le n.$$
\end{Definition}
\noindent In other words, considering (\ref{Qab}) we get
$$\varrho_1(X_{a\,b})\,e_{j_{1}\dots j_{n}} = q_{j_{a}j_{b}}\,e_{j_{1}\dots j_{n}}.$$

\begin{Definition}\label{ro2}
We define a linear operator $\varrho_2 \colon {\mathbb{C}}[S_n] \to \textnormal{End}({\mathcal{B}}_{Q})$ by
$$\varrho_2(g)\,e_{j_{1}\dots j_{n}}:=e_{j_{g^{-1}(1)}\dots j_{g^{-1}(n)}}$$
for every $g\in S_{n}$.
\end{Definition}

\begin{Proposition}\label{homro2}
A map $\varrho_2 \colon {\mathbb{C}}[S_n] \to \textnormal{End}({\mathcal{B}}_{Q})$ is a representation.
\end{Proposition}
\begin{proof} 
In fact $\varrho_2$ is a (right) regular representation.
\end{proof}

Let $\varrho \colon {\mathcal{A}(S_{n})} \to \textrm{End}({\mathcal{B}}_{Q})$ be a map defined on decomposable elements by
\begin{equation}\label{ro} 
\varrho(pg):=\varrho_1(p)\cdot \varrho_2(g)
\end{equation}
for every $p\in R_{n}$ and $g\in S_{n}$ and extended by additivity. In the trivial cases we have
\begin{itemize}
\item [$(i)$] \quad $\varrho(1\cdot g)\,e_{j_{1}\dots j_{n}} = \varrho_1(1)\cdot \varrho_2(g)\,e_{j_{1}\dots j_{n}} = 1\cdot e_{j_{g^{-1}(1)}\dots j_{g^{-1}(n)}} = e_{j_{g^{-1}(1)}\dots j_{g^{-1}(n)}}$,
\item [$(ii)$] \quad $\varrho(X_{a\,b}\,e)\,e_{j_{1}\dots j_{n}} = \varrho_1(X_{a\,b})\cdot \varrho_2(e)\,e_{j_{1}\dots j_{n}} =  Q_{a\,b}\,e_{j_{1}\dots j_{n}} = q_{j_{a}j_{b}}\,e_{j_{1}\dots j_{n}}$.
\end{itemize}

\begin{Remark}\label{tmul}
By applying the formula $(\ref{ro})$ on the general elements of the twisted group algebra we get$:$
\begin{align*}
\varrho\left(\sum_{g_{i}\in S_{n}} p_{i}\,g_{i}\right)\,e_{j_{1}\dots j_{n}} &= \sum_{g_{i}\in S_{n}} \varrho(p_{i}\,g_{i})\,e_{j_{1}\dots j_{n}}\\ 
&= \sum_{g_{i}\in S_{n}} \varrho_1\left( p_{i}(\dots, X_{a\,b},\dots)\right)\cdot \left(\varrho_2(g_{i})\,e_{j_{1}\dots j_{n}}\right)\\
&= \sum_{g_{i}\in S_{n}} p_{i} \left( \dots, q_{j_{g_{i}^{-1}(a)}j_{g_{i}^{-1}(b)}},\dots \right)e_{j_{g_{i}^{-1}(1)}\dots j_{g_{i}^{-1}(n)}}.
\end{align*}
Note that the basic instance of the multiplication $(\ref{mtwga})$ in ${\mathcal{A}(S_{n})}$ is given by the following formula
\begin{equation}\label{mulgen}
\left( X_{a\,b}\,g_{1} \right) \cdot \left( X_{c\,d}\,g_{2} \right) = \left( X_{a\,b}\cdot X_{g_{1}(c)\,g_{1}(d)} \right)\, g_{1}g_{2}
\end{equation}
which are the consequences of the following two types of basic relations:
\begin{equation}\label{rel1}
X_{a\,b}\cdot X_{c\,d} = X_{c\,d}\cdot X_{a\,b},
\end{equation}
\begin{equation}\label{rel2}
        g{\bf.}X_{a\,b} = X_{g(a)\,g(b)}\, g.
\end{equation}
\end{Remark}
\begin{Proposition}\label{alhom}
A map $\varrho \colon {\mathcal{A}(S_{n})} \to \textnormal{End}({\mathcal{B}}_{Q})$ is a representation.
\end{Proposition}
\begin{proof}
It is enough to check that $\varrho$ preserves the basic relations (\ref{rel1}) and (\ref{rel2}), where we will apply the formula (\ref{ro}) and Definitions~\ref{ro1} and \ref{ro2}.\\ 
It is easy to see that:
\begin{itemize}
\item [$(i)$] \quad $\varrho(X_{a\,b}\cdot X_{c\,d}) = Q_{a\,b}\cdot Q_{c\,d} = Q_{c\,d}\cdot Q_{a\,b} = \varrho(X_{c\,d}\cdot X_{a\,b})$.
\item [$(ii)$] \quad Now we will show that \, $\varrho(g{\bf.}X_{a\,b})\,e_{j_{1}\dots j_{n}} = \varrho(X_{g(a)\,g(b)}\, g)\,e_{j_{1}\dots j_{n}}$.\\
Note that
\begin{align*}
L \equiv \varrho(g{\bf.}X_{a\,b})\,e_{j_{1}\dots j_{n}} &= \varrho_2(g)\cdot  \varrho_1(X_{a\,b})\,e_{j_{1}\dots j_{n}} = \varrho_2(g)\, q_{j_{a}j_{b}}\,e_{j_{1}\dots j_{n}}\\
&= q_{j_{a}j_{b}}\, \varrho_2(g)\,e_{j_{1}\dots j_{n}}\\ 
&= q_{j_{a}j_{b}}\,e_{j_{g^{-1}(1)}\dots j_{g^{-1}(n)}};\\
D \equiv \varrho(X_{g(a)\,g(b)}\, g)\,e_{j_{1}\dots j_{n}} &= \varrho_1(X_{g(a)\,g(b)})\cdot \varrho_2(g)\,e_{j_{1}\dots j_{n}}\\
&= Q_{g(a)\,g(b)}\,e_{j_{g^{-1}(1)}\dots j_{g^{-1}(n)}}\\ 
&= q_{j_{g^{-1}(g(a))}j_{g^{-1}(g(b))}}\,e_{j_{g^{-1}(1)}\dots j_{g^{-1}(n)}}\\
&= q_{j_{a}j_{b}}\,e_{j_{g^{-1}(1)}\dots j_{g^{-1}(n)}}.
\end{align*}
\end{itemize}
\end{proof} 
\noindent In the case $Q$ is a set we call the representation $\varrho$ in Proposition~\ref{alhom} a twisted regular representation.

\begin{Lemma}\label{rog}
The representation $\varrho$ applied to element $\displaystyle g^*=\left(\prod_{(a,b)\in I(g^{-1})} X_{a\,b} \right) g$\, is given by
\begin{equation}\label{rogtilda} 
         \varrho(g^*)\,e_{j_{1}\dots j_{n}}=\prod_{(a,b)\in I(g)} q_{j_{b}j_{a}}\,e_{j_{g^{-1}(1)}\dots j_{g^{-1}(n)}}.
\end{equation}
\end{Lemma}
\begin{proof}
Note that by applying the formula (\ref{ro}) on $g^*=\left(\prod_{(a',b')\in I(g^{-1})} X_{a'\,b'} \right)\, g$ we obtain
\begin{align*}
\varrho \left( g^* \right)\,e_{j_{1}\dots j_{n}} &= \prod_{(a',b')\in I(g^{-1})} \varrho_1\left( X_{a'\,b'} \right)\cdot \varrho_2(g)\,e_{j_{1}\dots j_{n}}\\
&=\prod_{(a',b')\in I(g^{-1})} q_{j_{g^{-1}(a')}j_{g^{-1}(b')}}\,e_{j_{g^{-1}(1)}\dots j_{g^{-1}(n)}}\\
&=\prod_{(b,a)\in I(g)} q_{j_{a}j_{b}}\,e_{j_{g^{-1}(1)}\dots j_{g^{-1}(n)}}=\prod_{(a,b)\in I(g)} q_{j_{b}j_{a}}\,e_{j_{g^{-1}(1)}\dots j_{g^{-1}(n)}}
\end{align*}
with $a=g^{-1}(a')$,\, $b=g^{-1}(b')$.\\ 
Now it is easy to check that $(a',b')\in I(g^{-1})$ implies\, $a'< b'$, $g^{-1}(a') > g^{-1}(b')$ i.e\, $g(a)< g(b)$, $a > b$. Thus we get $(b,a)\in I(g)$.
\end{proof}
\begin{Remark}\label{rotbaa}
A direct consequence of the Lemma~$\ref{rog}:$\, the element $\varrho(t_{b,a}^*)\in \textnormal{End}({\mathcal{B}}_{Q})$ is given by
$$\varrho(t_{b,a}^*)\,e_{j_{1}\dots j_{a}j_{a+1}\dots j_{b}\dots j_{n}} = \prod_{a \le i \le b-1} q_{j_{b}j_{i}}\,e_{j_{1}\dots j_{b}j_{a}\dots j_{b-1}\dots j_{n}}$$
and in special case \hspace{25pt} $\varrho(t_{a}^*)\,e_{j_{1}\dots j_{a}j_{a+1}\dots j_{n}} = q_{j_{a+1}j_{a}}\,e_{j_{1}\dots j_{a+1}j_{a}\dots j_{n}}$\\ 
$($recall\, $t_{a}^*=t_{a+1,a}^*).$\\
From the identity $(\ref{sqta})$ we obtain:
$$\varrho((t_{a}^*)^2)\,e_{j_{1}\dots j_{n}} = \sigma_{j_{a}j_{a+1}}\,e_{j_{1}\dots j_{n}},$$
where we denoted \, $\sigma_{j_{a}j_{a+1}}:=q_{j_{a}j_{a+1}} q_{j_{a+1}j_{a}}.$
\end{Remark}
Note that the matrices of the operators in the twisted regular representation of ${\mathcal{A}(S_{n})}$ on the generic (resp.~degenerate) weight subspaces ${\mathcal{B}}_{Q}\subset {\mathcal{B}}$ are square matrices of size $n!$ (resp.~$\textrm{Card}~\widehat{Q}$) whose entries are monomials (resp.~polynomials) in $q_{ij}$'s.\\

\noindent Recall that the element ${\alpha^*_{n}}\in {\mathcal{A}(S_{n})}$ is given by $\displaystyle{\alpha^*_{n}} = \sum_{g\in S_{n}} g^*, \, n\ge 1.$
\begin{Proposition}\label{roAQgen}
Let\, $\varrho \colon {\mathcal{A}(S_{n})} \to \textnormal{End}({\mathcal{B}}_{Q})$ be the twisted regular representation on the generic weight space ${\mathcal{B}}_{Q}$. Then the $(\underline{k},\underline{j})$-entry of the matrix ${\mathbf{A}}_{Q}$ of the $\varrho({\alpha^*_{n}})$ is given by
\begin{equation}\label{AQgen} 
         \left( {\mathbf{A}}_{Q} \right)_{\underline{k},\underline{j}} = \prod_{(a,b)\in I(g)} q_{j_{b}j_{a}},
\end{equation}
where $g$ satisfies\, $\underline{k}=g{\bf.}\underline{j}$ \, $(\underline{j}=j_{1}\dots j_{n}\in \widehat{Q}$, $\underline{k}=k_{1}\dots k_{n}\in \widehat{Q})$. 
\end{Proposition}
\begin{proof}
By applying Lemma~\ref{rog} we see that $\varrho({\alpha^*_{n}})\in \textnormal{End}({\mathcal{B}}_{Q})$ acts as
$$\varrho({\alpha^*_{n}})\,e_{j_{1}\dots j_{n}}=\sum_{g\in S_{n}} \left( \prod_{(a,b)\in I(g)} q_{j_{b}j_{a}}\,e_{j_{g^{-1}(1)}\dots j_{g^{-1}(n)}}\right)$$
i.e
$$\varrho({\alpha^*_{n}})\,e_{\underline{j}}=\sum_{g\in S_{n}} \left( \prod_{(a,b)\in I(g)} q_{j_{b}j_{a}}\,e_{\underline{k}}\right).$$
Therefore we have that the $(\underline{k},\underline{j})$-entry of ${\mathbf{A}}_{Q}$ is equal to $\displaystyle\prod_{(a,b)\in I(g)} q_{j_{b}j_{a}}$, so the identity (\ref{AQgen}) follows.
\end{proof}
\noindent In particular, if $\textrm{Card}~Q=1$, then ${\mathbf{A}}_{Q}=1$. Thus, we suppose that $\textrm{Card}~Q=n\ge 2$.  
\begin{Remark}\label{roAQdeg} 
If\, $\varrho \colon {\mathcal{A}(S_{n})} \to \textnormal{End}({\mathcal{B}}_{Q})$ is the representation $($on the degenerate weight subspace ${\mathcal{B}}_{Q})$, then the $(\underline{k},\underline{j})$-entry of ${\mathbf{A}}_{Q}$ is given by
\begin{equation}\label{AQdeg} 
         \left( {\mathbf{A}}_{Q} \right)_{\underline{k},\underline{j}}=\sum_{g\in g(\underline{k},\underline{j})} \left( \prod_{(a,b)\in I(g)} q_{j_{b}j_{a}} \right),
\end{equation}
where\, $g(\underline{k},\underline{j}):=\{ g\in S_{n} \mid k_{p}=j_{g^{-1}(p)} \mbox{\, for all\, } 1\le p\le n\}$.
\end{Remark} 

\subsection{Factorization of the matrix ${\mathbf{A}}_{Q}$}

\noindent Let us denote\, ${\mathbf{T}}_{b,a} := \varrho(t_{b,a}^*)$, ${\mathbf{T}}_{a} := \varrho(t_{a}^*)$. If $b=a$ then ${\mathbf{T}}_{b,a} = \mathbf{I}$. The $(\underline{k},\underline{j})$-entry of the corresponding matrices ${\mathbf{T}}_{b,a}, \, 1\le a < b\le n$ and ${\mathbf{T}}_{a}, \, 1\le a\le n-1$ are given by 
\begin{equation}\label{entTba} 
 \left( {\mathbf{T}}_{b,a} \right)_{\underline{k},\underline{j}}= \left\{ 
   \begin{array}{cl}      
         \displaystyle\prod_{a\le i \le b-1} q_{j_{b}j_{i}} & \textrm{if\, } \underline{k}=t_{b,a}{\bf.}\underline{j}\\
          0 & \textrm{otherwise}
    \end{array} \right.     
\end{equation}
with \, $t_{b,a}{\bf.}\underline{j} = j_{1}\dots j_{b}j_{a}\dots j_{b-1}\dots j_{n}$\, and 
\begin{equation}\label{entTa} 
 \left( {\mathbf{T}}_{a} \right)_{\underline{k},\underline{j}}= \left\{ 
   \begin{array}{cl}      
         q_{j_{a+1}j_{a}} & \textrm{if\, } \underline{k}=t_{a}{\bf.}\underline{j}\\
          0 & \textrm{otherwise}
    \end{array} \right.     
\end{equation}
with \, $t_{a}{\bf.}\underline{j} = j_{1}\dots j_{a+1}j_{a}\dots j_{n}$ (see Remark~\ref{rotbaa}). Now it is easy to see that
\begin{equation}\label{entTa2} 
({\mathbf{T}}_{a})^2\,e_{\underline{j}}=\sigma_{j_{a}j_{a+1}}\,e_{\underline{j}} 
\end{equation}
is the diagonal matrix with $\sigma_{j_{a}j_{a+1}}$ as $\underline{j}$-$th$ diagonal entry.\\

\noindent Now we consider the elements ${\beta^*_{k}}\in {\mathcal{A}(S_{n})}$,\, $1\le k\le n$ defined before (\ref{alphafac}) in Section 3. The corresponding elements $\varrho({\beta^*_{n-k+1}}) \in \textnormal{End}({\mathcal{B}}_{Q})$ are given by 
\begin{align*}
\varrho({\beta^*_{n-k+1}})\,e_{\underline{j}} &= \varrho \left( \sum_{k\le m\le n} t_{m,k}^* \right) e_{\underline{j}} = \left(\sum_{k\le m\le n} \varrho(t_{m,k}^*)\right) e_{\underline{j}}\\
&= \left(\sum_{k+1\le m\le n} \varrho(t_{m,k}^*) + \varrho(id) \right) e_{\underline{j}}
\end{align*}
i.e
\begin{equation}\label{robetank1}
\varrho({\beta^*_{n-k+1}})\,e_{\underline{j}} = \sum_{k+1\le m\le n} \varrho(t_{m,k}^*)\,e_{\underline{j}} + e_{\underline{j}}
\end{equation}
(see Remark~\ref{rotbaa} and also Definition~\ref{ro2}). Let 
$${\mathbf{B}}_{Q,n-k+1} := \varrho({\beta^*_{n-k+1}}), \quad 1\le k\le n-1,$$ 
where \, ${\mathbf{B}}_{Q,1}=\varrho({\beta^*_{1}})=\varrho(id)={\mathbf{I}}$. Then in the matrix notation (\ref{robetank1}) can be written  $$\displaystyle{\mathbf{B}}_{Q,n-k+1} = \sum_{k+1\le m\le n} {\mathbf{T}}_{m,k} + {\mathbf{I}}$$ 
or shorter \quad $\displaystyle{\mathbf{B}}_{Q,n-k+1} = \sum_{k\le m\le n} {\mathbf{T}}_{m,k}$.\\ 

\noindent Note that the $(\underline{k},\underline{j})$-entry of ${\mathbf{B}}_{Q,n-k+1}$\, ($1\le k \le n-1$) is equal to $\displaystyle\prod_{k\le i \le m-1} q_{j_{m}j_{i}}$ if\, $\underline{k} = t_{m,k}{\bf.}\underline{j} = j_{1}\dots j_{m}j_{k}\dots j_{m-1}\dots j_{n}$,\, otherwise it is equal to zero.\\ 

\noindent Thus we can write
\begin{equation}\label{entBQn-k+1} 
 \left( {\mathbf{B}}_{Q,n-k+1} \right)_{\underline{k},\underline{j}}= \left\{ 
   \begin{array}{cl}      
         \displaystyle\prod_{k\le i < m} q_{j_{m}j_{i}} & \textrm{if\, } \underline{k}=t_{m,k}{\bf.}\underline{j} \quad k\le m\le n\\
          0 & \textrm{otherwise}
    \end{array} \right.     
\end{equation}
for each $1\le k \le n-1$.\\ 
In the special case for\, $m=k$\, all $\underline{j}$-$th$ diagonal entries of ${\mathbf{B}}_{Q,n-k+1}$ are equal to one.
\begin{Remark}\label{matBQn}
For $k=1$ we have 
$${\mathbf{B}}_{Q,n} = \sum_{1\le m\le n} {\mathbf{T}}_{m,1} = {\mathbf{T}}_{n,1} + {\mathbf{T}}_{n-1,1} +\cdots + {\mathbf{T}}_{3,1} + {\mathbf{T}}_{2,1} + {\mathbf{I}}$$
$($where\, ${\mathbf{T}}_{1,1} = {\mathbf{I}})$ so the $(\underline{k},\underline{j})$-entry of\, ${\mathbf{B}}_{Q,n}$\, is given by
\begin{equation*} 
 \left( {\mathbf{B}}_{Q,n} \right)_{\underline{k},\underline{j}}= \left\{ 
   \begin{array}{cl}      
         \displaystyle q_{j_{m}j_{1}}\cdots q_{j_{m}j_{m-1}} & \textrm{if\, } \underline{k}=t_{m,1}{\bf.}\underline{j} \quad 1\le m\le n\\
          0 & \textrm{otherwise}
    \end{array} \right.     
\end{equation*}
$($with\, $\underline{k} = j_{m}j_{1}\dots j_{m-1}j_{m+1}\dots j_{n})$. In other words we get the following identity
\begin{equation*}
{\mathbf{B}}_{Q,n}\,e_{\underline{j}} = \sum_{1\le m\le n} q_{j_{m}j_{1}}\cdots q_{j_{m}j_{m-1}}\,e_{j_{m}j_{1}\dots j_{m-1}j_{m+1}\dots j_{n}}
\end{equation*}
$($compare with $(\ref{partQ}))$. Now it is easy to see that the matrix ${\mathbf{B}}_{Q,n}$ is equal to the matrix ${\textrm{B}_Q}$ $($i.e the matrix of ${\partial^Q}$ w.r.t monomial basis of ${{\mathcal{B}}_{Q}}\subset {\mathcal{B}}$; see first section$)$. It turns out that the factorization of the matrix ${\mathbf{B}}_{Q,n}$ is equivalent to factorization of ${\textrm{B}_Q}$, so the problem of computing $\det{\textrm{B}_{Q}}$ can be reduced to the problem of computing $\det{{\mathbf{B}}_{Q,n}}$. With this motivation we are going to find a formula for the factorization of ${\mathbf{B}}_{Q,n}$ and also its determinant.
\end{Remark}
\noindent In what follows we will consider the additional elements ${\gamma^*_{k}}$, ${\delta^*_{k}}$, $1\le k\le n$ in the algebra ${\mathcal{A}(S_{n})}$ defined after (\ref{alphafac}) in Section 3. The corresponding elements $\varrho({\gamma^*_{n-k+1}}), \varrho({\delta^*_{n-k+1}}) \in \textnormal{End}({\mathcal{B}}_{Q})$, $1\le k\le n-1$ are given by
\begin{align*}
\varrho({\gamma^*_{n-k+1}})\,e_{\underline{j}} &=\left( id-\varrho(t_{n,k}^*)\right)\cdot \left(id-\varrho(t_{n-1,k}^*)\right)\cdots \left(id -\varrho(t_{k+1,k}^*)\right) e_{\underline{j}}\\
\varrho({\delta^*_{n-k+1}})\,e_{\underline{j}} &= \left( id-\varrho((t_{k}^*)^{2})\, \varrho(t_{n,k+1}^*)\right)\cdot \left( id-\varrho((t_{k}^*)^{2})\, \varrho(t_{n-1,k+1}^*)\right)\cdots\\ 
&\hspace{20pt} \left( id-\varrho((t_{k}^*)^{2})\, \varrho(t_{k+2,k+1}^*)\right)\cdot \left( id-\varrho((t_{k}^*)^{2})\right) e_{\underline{j}}
\end{align*}
which in matrix notation corresponds to the following expressions 
\begin{align*}
{\mathbf{C}}_{Q,n-k+1} &=\left( {\mathbf{I}}-{\mathbf{T}}_{n,k}\right)\cdot \left( {\mathbf{I}}-{\mathbf{T}}_{n-1,k}\right)\cdots \left( {\mathbf{I}}-{\mathbf{T}}_{k+1,k}\right)\\
{\mathbf{D}}_{Q,n-k+1} &=\left( {\mathbf{I}}-({\mathbf{T}}_{k})^{2}\, {\mathbf{T}}_{n,k+1}\right)\cdot \left( {\mathbf{I}}-({\mathbf{T}}_{k})^{2}\, {\mathbf{T}}_{n-1,k+1}\right)\cdots \left( {\mathbf{I}}-({\mathbf{T}}_{k})^{2}\right)
\end{align*}
Here we have introduced notations
$${\mathbf{C}}_{Q,n-k+1} := \varrho({\gamma^*_{n-k+1}}),\quad {\mathbf{D}}_{Q,n-k+1} := \varrho({\delta^*_{n-k+1}}),$$
$1\le k\le n-1$. Clearly, $({\mathbf{T}}_{k})^{2} = ({\mathbf{T}}_{k+1,k})^{2}$ is the diagonal matrix given by (\ref{entTa2}).\\
By using the identity (\ref{betafac}) we obtain
$${\mathbf{B}}_{Q,n-k+1} = {\mathbf{D}}_{Q,n-k+1} \cdot \left( {\mathbf{C}}_{Q,n-k+1}\right)^{-1} \hspace{15pt} \textrm{for all} \hspace{10pt} 1\le k\le n-1.$$
and more precisely
\begin{align*}
{\mathbf{B}}_{Q,n-k+1} =& \left( {\mathbf{I}}-({\mathbf{T}}_{k})^{2}\, {\mathbf{T}}_{n,k+1}\right)\cdot \left( {\mathbf{I}}-({\mathbf{T}}_{k})^{2}\, {\mathbf{T}}_{n-1,k+1}\right)\cdots \left( {\mathbf{I}}-({\mathbf{T}}_{k})^{2}\, {\mathbf{T}}_{k+2,k+1} \right)\\
&\cdot \left( {\mathbf{I}}-({\mathbf{T}}_{k})^{2}\right)\cdot \left( {\mathbf{I}}-{\mathbf{T}}_{k+1,k}\right)^{-1}\cdots \left( {\mathbf{I}}-{\mathbf{T}}_{n-1,k}\right)^{-1}\cdot \left( {\mathbf{I}}-{\mathbf{T}}_{n,k}\right)^{-1}
\end{align*}
or in shorter form
\begin{equation}\label{BQn-k+1fac}
{\mathbf{B}}_{Q,n-k+1} = \prod_{k+1\le m\le n}^{\gets} \left( {\mathbf{I}}-({\mathbf{T}}_{k})^{2}\, {\mathbf{T}}_{m,k+1} \right)\cdot \prod_{k+1\le m\le n}^{\to} \left( {\mathbf{I}}-{\mathbf{T}}_{m,k} \right)^{-1}.
\end{equation}
On the other hand, from the identity (\ref{alphafac}) in ${\mathcal{A}(S_{n})}$ we get
$${\mathbf{A}}_{Q} = \prod_{1\le k\le n-1}^{\gets} \left( {\mathbf{D}}_{Q,n-k+1} \cdot \left( {\mathbf{C}}_{Q,n-k+1}\right)^{-1} \right)$$
i.e
\begin{equation}\label{AQfac}
{\mathbf{A}}_{Q} =\prod_{1\le k\le n-1}^{\gets} \left ( \prod_{k+1\le m\le n}^{\gets} \left( {\mathbf{I}}-({\mathbf{T}}_{k})^{2}\, {\mathbf{T}}_{m,k+1} \right)\cdot \prod_{k+1\le m\le n}^{\to} \left( {\mathbf{I}}-{\mathbf{T}}_{m,k} \right)^{-1}\right).
\end{equation}
Now it is easy to see that for computing $\det{\mathbf{B}}_{Q,n-k+1}$ and $\det{\mathbf{A}}_{Q}$ it is enough to compute $\det({\mathbf{I}}-{\mathbf{T}}_{b,a})$ and $\det({\mathbf{I}}-({\mathbf{T}}_{a-1})^{2}\, {\mathbf{T}}_{b,a})$ ($1\le a < b\le n$).\\

\noindent Let us denote
\begin{align*}
Q \choose m &= \{ T\subseteq Q \mid \textrm{Card}~T = m\},\\
\sigma_T &= \prod_{\{i\ne j \}\subset T} \sigma_{ij} = \prod_{i\ne j \in T} q_{ij}.
\end{align*} 
\begin{Lemma}\label{factdet}
Let $\varrho \colon {\mathcal{A}(S_{n})} \to \textnormal{End}({\mathcal{B}}_{Q})$ be the twisted regular representation of twisted group algebra ${\mathcal{A}(S_{n})}$ on any generic subspace ${\mathcal{B}}_{Q}$ of the algebra ${\mathcal{B}}$. Then
\begin{itemize}
\item [$(i)$] \quad $\displaystyle \det({\mathbf{I}}-{\mathbf{T}}_{b,a}) = \prod_{T\in {Q \choose b-a+1}} (1-\sigma_T)^{(b-a)!\cdot (n-b+a-1)!}$ \quad $(1\le a < b\le n)$\\
\item [$(ii)$] \quad $\displaystyle \det({\mathbf{I}}-({\mathbf{T}}_{a-1})^{2}\, {\mathbf{T}}_{b,a}) = \prod_{T\in {Q \choose b-a+2}} (1-\sigma_T)^{(b-a)!\cdot (b-a+2)\cdot (n-b+a-2)!}$\\
$(1 < a \le b\le n)$.
\end{itemize} 
\end{Lemma}
\noindent Note that this Lemma is the twisted group algebra analogue of the Lemma~1.9.1 in the paper of Svrtan and Meljanac (see \cite{3}). Therefore the proof will be similar to the proof of Lemma 1.9.1. Here we use the factorizations in different direction.  
\begin{proof}
\begin{itemize}
\item [$(i)$] \quad Let $H:=\langle t_{b,a} \rangle\subset S_{n}$ be the cyclic subgroup of $S_{n}$ generated by the cycle $t_{b,a}$ (whose lenght is equal to $b-a+1$). Then every $H-$orbit on generic subspace ${\mathcal{B}}_{Q}$ is given by
$${\mathcal{B}}_{Q}^{[\underline{j}]_a^b} = span_{{\mathbb{C}}} \left\{ e_{t_{b,a}^k{\bf.}\underline{j}} \mid 0\le k\le b-a \right\},$$
which corresponds to a cyclic $t_{b,a}-$equivalence class\\ 
$[\underline{j}]_a^b = j_1j_2\dots (j_aj_{a+1}\dots j_b)\dots j_n$ \, of the sequence $\underline{j}=j_1\dots j_n\in \widehat{Q}$.\\ 
We get
$${\mathbf{T}}_{b,a}\left( e_{t_{b,a}^k{\bf.}\underline{j}} \right) = c_k\,e_{t_{b,a}^{k+1}{\bf.}\underline{j}}, \hspace{30pt} 0\le k\le b-a,$$
where
\begin{align*}
c_{0} &= q_{j_{b}j_{a}}q_{j_{b}j_{a+1}}q_{j_{b}j_{a+2}}\cdots q_{j_{b}j_{b-1}},\\
c_{1} &= q_{j_{b-1}j_{b}}q_{j_{b-1}j_{a}}q_{j_{b-1}j_{a+1}}\cdots q_{j_{b-1}j_{b-2}},\\
c_{2} &= q_{j_{b-2}j_{b-1}}q_{j_{b-2}j_{b}}q_{j_{b-2}j_{a}}\cdots q_{j_{b-2}j_{b-3}},\\
\vdots \\
c_{b-a-1} &= q_{j_{a+1}j_{a+2}}q_{j_{a+1}j_{a+3}}q_{j_{a+1}j_{a+4}}\cdots q_{j_{a+1}j_{a}},\\
c_{b-a} &= q_{j_{a}j_{a+1}}q_{j_{a}j_{a+2}}q_{j_{a}j_{a+3}}\cdots q_{j_{a}j_{b}}.
\end{align*}
By using (\ref{entTba}) and by applying the identities given above, we obtain the following
\begin{equation}\label{potTba} 
\left( {\mathbf{T}}_{b,a} \right)^k\,e_{\underline{j}} = \prod_{0\le i\le k-1} c_{i}\,e_{t_{b,a}^k{\bf.}\underline{j}}, \hspace{25pt} 1\le k\le b-a+1.
\end{equation}
\noindent By considering that\, ${\mathbf{T}}_{b,a}|{\mathcal{B}}_{Q}^{[\underline{j}]_a^b}$ is a cyclic operator (which corresponds to cyclic matrix ${\mathbf{T}}_{b,a}$) we can write
\begin{align*}
\det \left( ({\mathbf{I}}-{\mathbf{T}}_{b,a})|{\mathcal{B}}_{Q}^{[\underline{j}]_a^b}\right) &= 1- \prod_{0\le i\le b-a} c_{i}\\ 
&= 1- \prod_{i\neq j\in T} q_{ij} = 1- \prod_{\{i,j\}\subset T} \sigma_{ij}
\end{align*}
i.e
\begin{equation}\label{detIminT}
\det \left( ({\mathbf{I}}-{\mathbf{T}}_{b,a})|{\mathcal{B}}_{Q}^{[\underline{j}]_a^b}\right) = 1-\sigma_T
\end{equation}
where\, $\sigma_{ij}=q_{ij}q_{ji}$,\, $T = \{j_a,\dots, j_b \}$ ($1\le a < b\le n$)\, and\, $\textrm{Card}~T = b-a+1$.\\ 
Now we give that there are $(b-a)!\cdot(n-(b-a+1))! = (b-a)!\cdot(n-b+a-1)!$\, $H-$orbits for which determinant (\ref{detIminT}) gets the value of\, $1-\sigma_T$, therefore we obtain \, 
$\det({\mathbf{I}}-{\mathbf{T}}_{b,a}) = \prod_{T\in {Q \choose b-a+1}} (1-\sigma_T)^{(b-a)!\cdot (n-b+a-1)!}$.\\
\item [$(ii)$] \quad Will be proven in a manner similar to $(i)$, where we have
$$({\mathbf{T}}_{a-1})^{2}\, {\mathbf{T}}_{b,a}\left( e_{t_{b,a}^k{\bf.}\underline{j}} \right) = d_k\,e_{t_{b,a}^{k+1}{\bf.}\underline{j}}, \hspace{30pt} 0\le k\le b-a,$$
with
\begin{align*}
d_{0} &= \sigma_{j_{a-1}j_{b}} c_{0}= \sigma_{j_{a-1}j_{b}} q_{j_{b}j_{a}}q_{j_{b}j_{a+1}}q_{j_{b}j_{a+2}}\cdots q_{j_{b}j_{b-1}},\\
d_{1} &=\sigma_{j_{a-1}j_{b-1}} c_{1}= \sigma_{j_{a-1}j_{b-1}} q_{j_{b-1}j_{b}}q_{j_{b-1}j_{a}}q_{j_{b-1}j_{a+1}}\cdots q_{j_{b-1}j_{b-2}},\\
d_{2} &=\sigma_{j_{a-1}j_{b-2}} c_{2}= \sigma_{j_{a-1}j_{b-2}} q_{j_{b-2}j_{b-1}}q_{j_{b-2}j_{b}}q_{j_{b-2}j_{a}}\cdots q_{j_{b-2}j_{b-3}},\\
\vdots \\
d_{b-a-1} &=\sigma_{j_{a-1}j_{a+1}} c_{b-a+1}= \sigma_{j_{a-1}j_{a+1}} q_{j_{a+1}j_{a+2}}q_{j_{a+1}j_{a+3}}q_{j_{a+1}j_{a+4}}\cdots q_{j_{a+1}j_{a}},\\
d_{b-a} &=\sigma_{j_{a-1}j_{a}} c_{b-a}= \sigma_{j_{a-1}j_{a}} q_{j_{a}j_{a+1}}q_{j_{a}j_{a+2}}q_{j_{a}j_{a+3}}\cdots q_{j_{a}j_{b}}.
\end{align*}
Here we have
\begin{equation}\label{potTa-1Tba} 
\left( ({\mathbf{T}}_{a-1})^{2}\, {\mathbf{T}}_{b,a} \right)^k\,e_{\underline{j}} = \prod_{0\le i\le k-1} d_{i}\,e_{t_{b,a}^k{\bf.}\underline{j}}, \hspace{25pt} 1\le k\le b-a+1.
\end{equation}
and
\begin{align*}
\det \left( ({\mathbf{I}}-({\mathbf{T}}_{a-1})^{2}\, {\mathbf{T}}_{b,a})|{\mathcal{B}}_{Q}^{[\underline{j}]_a^b}\right) &= 1- \prod_{0\le i\le b-a} d_{i}\\ 
&= 1- \prod_{i\neq j\in T} q_{ij} = 1- \prod_{\{i,j\}\subset T} \sigma_{ij}
\end{align*}
i.e
\begin{equation}\label{detIminTT}
\det \left( ({\mathbf{I}}-({\mathbf{T}}_{a-1})^{2}\, {\mathbf{T}}_{b,a})|{\mathcal{B}}_{Q}^{[\underline{j}]_a^b}\right) = 1-\sigma_T
\end{equation}
where\, $T = \{j_{a-1}j_a,\dots, j_b \}$ ($1 < a \le b\le n$)\, and\, $\textrm{Card}~T = b-a+2$. 
There are\, $(b-a)!\cdot(b-a+2)\cdot(n-b+a-2)!$\, $H-$orbits for which determinant (\ref{detIminTT}) gets the value of\, $1-\sigma_T$, therefore we have proved $(ii)$.
\end{itemize} 
\end{proof}
\begin{Theorem}\label{factdetAQB}
Let $\varrho \colon {\mathcal{A}(S_{n})} \to \textnormal{End}({\mathcal{B}}_{Q})$ be a twisted regular representation $($where ${\mathcal{B}}_{Q}$ is generic subspace of ${\mathcal{B}})$. Then we have
\begin{itemize}
\item [$(i)$] \quad $\displaystyle\det{\mathbf{A}}_{Q} = \prod_{2 \le m\le n}\, \prod_{T\in {Q \choose m}} (1-\sigma_T)^{(m-2)!\cdot (n-m+1)!}$,\\
\item [$(ii)$] \quad $\displaystyle\det{\mathbf{B}}_{Q,n-k+1} = \prod_{2 \le m\le n-k+1}\, \prod_{T\in {Q \choose m}} (1-\sigma_T)^{(m-2)!\cdot (n-m)!}$,\quad $(1 < k \le n-1)$.

\end{itemize} 
\end{Theorem}
\noindent This Theorem is similar to Theorem~1.9.2 in \cite{3}.
\begin{proof}
By using Lemma~\ref{factdet} we get the following:
\begin{align*}
\det{\mathbf{C}}_{Q,n-k+1} &= \prod_{k+1 \le p\le n} \det({\mathbf{I}}-{\mathbf{T}}_{p,k})\\ 
&= \prod_{k+1 \le p\le n}\, \prod_{T\in {Q \choose p-k+1}} (1-\sigma_T)^{(p-k)!\cdot (n-p+k-1)!}\\
&= \prod_{2 \le m\le n-k+1}\, \prod_{T\in {Q \choose m}} (1-\sigma_T)^{(m-1)!\cdot (n-m)!}
\end{align*}
\begin{align*}
\det{\mathbf{D}}_{Q,n-k+1} &= \prod_{k+1 \le p\le n} \det({\mathbf{I}}-({\mathbf{T}}_{k})^{2}\, {\mathbf{T}}_{p,k+1})\\ 
&= \prod_{k+1 \le p\le n}\, \prod_{T\in {Q \choose p-k+1}} (1-\sigma_T)^{(p-k-1)!\cdot (p-k+1)\cdot (n-p+k-1)!}\\
&= \prod_{2 \le m\le n-k+1}\, \prod_{T\in {Q \choose m}} (1-\sigma_T)^{(m-2)!\cdot m\cdot (n-m)!}
\end{align*}
Therefore by applying the formula \, $\displaystyle\det{\mathbf{B}}_{Q,n-k+1}=\frac{\det{\mathbf{D}}_{Q,n-k+1}}{\det{\mathbf{C}}_{Q,n-k+1}}$\, (see also (\ref{BQn-k+1fac})) we get
the formula given in $(ii)$, i.e
$$\det{\mathbf{B}}_{Q,n-k+1} = \prod_{2 \le m\le n-k+1}\prod_{T\in {Q \choose m}} (1-\sigma_T)^{(m-2)!\cdot (n-m)!}.$$ 
On the other hand by applying \, $\displaystyle\det{\mathbf{A}}_{Q}=\prod_{1\le k\le n-1}^{\gets} \det{\mathbf{B}}_{Q,n-k+1}$\, (see also (\ref{AQfac})) we get
\begin{align*}
\det{\mathbf{A}}_{Q} &=\prod_{1\le k\le n-1}^{\gets} \, \prod_{2\le m\le n-k+1}\, \prod_{T\in {Q \choose m}} (1-\sigma_T)^{(m-2)!\cdot (n-m)!}\\
&= \prod_{2 \le m\le n}\, \prod_{T\in {Q \choose m}} (1-\sigma_T)^{(m-2)!\cdot (n-m)!\cdot (n-m+1)}
\end{align*}
i.e
$$\det{\mathbf{A}}_{Q} = \prod_{2 \le m\le n}\, \prod_{T\in {Q \choose m}} (1-\sigma_T)^{(m-2)!\cdot (n-m+1)!}.$$
\end{proof}

\noindent More results of this type will be given in a subsequent paper on computation of constants in multiparametric quon algebras by using a twisted group algebra approach (c.f \cite{8}), which is more efficient than the approach of the paper \cite{9}.

\end{document}